\documentclass[12pt,letterpaper]{amsart}
\usepackage{amsmath,amssymb,amsfonts,amsthm,amscd}
\usepackage{mathrsfs,latexsym,url,graphicx,setspace,color}
\newtheorem{theorem}{Theorem}
\newtheorem*{theorem*}{Theorem}
\newtheorem{proposition}{Proposition}

\theoremstyle{remark}
\newtheorem{remark}{Remark}

\newtheorem{example}{Example}

\newcommand{\T}{\mathbb{T}}
\newcommand{\G}{{\bf{G}}}

\newcommand{\abs}[1]{\left\lvert#1\right\rvert}
\newcommand{\norm}[1]{\lvert\lvert#1\rvert\rvert}

\newcommand{\hardy}{\mathcal{H}}

\newcommand{\hol}{\mathcal{O}}

\newcommand{\disc}{\mathbb{D}}
\newcommand{\C}{\mathbb{C}}

\newcommand{\disk}{\mathbb{D}}

\newcommand{\Smu}{\mathbf{S}_{\mu}}

\usepackage{hyperref}

\title[$L^p$ Regularity of Weighted Szeg\"o Projections]{$L^p$ Regularity of Weighted Szeg\"o Projections on the Unit Disc}

\author[S. Munasinghe]{Samangi Munasinghe}
\author[Y. E. Zeytuncu]{Yunus E. Zeytuncu}
\address{Department of Mathematics, Western Kentucky University, Bowling Green, Kentucky, 42101}
\address{Department of Mathematics and Statistics, University of Michigan-Dearborn, Dearborn, MI 48128}
\email{samangi.munasinghe@wku.edu}
\email{zeytuncu@umich.edu}


\begin{document}
\begin{abstract}
We present a family of weights on the unit disc for which the corresponding weighted Szeg\"o projection operators are irregular on $L^p$ spaces.
We further investigate the dual spaces of weighted Hardy spaces corresponding to this family.
\end{abstract}

\maketitle 

\section{Introduction}

\subsection{Classical setting} Let $\disc$ denote the unit disc in $\C$ and $\mathbb{T}$ the unit circle. Let $\mathcal{O}(\mathbb{D})$ denote the set of holomorphic functions on $\mathbb{D}$. For $1\leq p<\infty$, the ordinary Hardy space is defined as

\begin{displaymath}
\hardy^p(\mathbb{T}) = \left\{  f \in \hol(\disk) ~\text{and}~ \norm f_{\hardy^p} < \infty  \right\},
\end{displaymath}
where 
\begin{displaymath}
\norm f_{\hardy^p}^{p} =  \underset{0 \leq r < 1}{\text{sup}} \int_{0}^{2\pi} \abs{f(re^{i\theta})}^p d\theta.
\end{displaymath}

It is known that (\cite{DurenHp70}) functions in $\hardy^p(\mathbb{T})$ have boundary limits almost everywhere, i.e., for almost every $\theta \in [0,2\pi]$
\begin{displaymath}
{f}(e^{i\theta}) = \underset{r \rightarrow 1^{-}}{\text{lim~}} f(re^{i\theta}) 
\end{displaymath}
exists.
Moreover, 
\begin{displaymath}
\norm {{f}}_{L^p(\mathbb{T})} =  \norm f_{\hardy^p(\mathbb{T})}  
\end{displaymath}
where $L^p(\mathbb{T})$ is defined using the standard Lebesgue measure (denoted by $d\theta$) on the unit circle.  
It is also known that $\hardy^p(\mathbb{T})$ is a closed subspace of $L^p(\mathbb{T})$. In particular, for $p=2$, the orthogonal projection operator, called the Szeg\"o projection operator exists; 
\begin{displaymath}
\mathbf{S}: L^2(\T) \longrightarrow \hardy^2(\T).
\end{displaymath}
The operator $\mathbf{S}$ is an integral operator with the kernel $S(z, w)$ (called the Szeg\"o kernel) and for $f\in L^2(\mathbb{T})$,
$$\mathbf{S}f(z)=\int_{\mathbb{T}}S(z,w)f(w)d\theta.$$

It follows from the general theory of reproducing kernels that for any orthonormal basis $\{e_n(z)\}_{n = 0}^{\infty}$ for $\hardy^2(\mathbb{T})$, the Szeg\"o kernel is given by 
\begin{displaymath}
S(z,w)=\sum_{n=0}^{\infty}e_n(z)\overline{e_n(w)}.
\end{displaymath}

\subsection{Weighted setting} 
Let $g(z)$ be a holomorphic function on $\mathbb{D}$ that is continuous on $\overline{\mathbb{D}}$ and has no zeros inside $\mathbb{D}$. We set $\mu(z)=|g(z)|^2$ and define weighted Hardy spaces and weighted Szeg\"o projections using the function $\mu(z)$ as a weight on $\T$.

For $1 \leq p < \infty$, we define the weighted Lebesgue and Hardy spaces with respect to $\mu$ as;
\begin{displaymath}
L^p(\T, \mu) = \lbrace  f \text{ measurable function on $\mathbb{D}$ and }  \norm{f}_{p,\mu} < \infty \rbrace
\end{displaymath}
where
\begin{eqnarray*}
\norm{f}_{p, \mu}^p = \int_{\T} \abs{f(w)}^p \mu(w) d\theta
= \int_{\T} \abs{f(w) (g(w))^{2/p}}^p d\theta
\end{eqnarray*}
and
\begin{displaymath}
\hardy^p(\T, \mu) = \lbrace f \in \mathcal{O}(\disk) \text{ such that } \norm{f}_{\hardy^p, \mu} < \infty\rbrace
\end{displaymath}
where
\begin{displaymath}
\norm{f}_{\hardy^p, \mu}^p = \underset{0 \leq r < 1}{\text{sup}}~~ \int_0^{2\pi} \abs{f(re^{i\theta}) (g(re^{i\theta}))^{2/p}}^p d\theta.
\end{displaymath}

Note that, $\displaystyle{f \in \hardy^p(\T, \mu)}$ implies $\displaystyle{f(z) (g(z))^{2/p} \in \hardy^p(\T)}$ and which in turn gives that  $\displaystyle{f(z) (g(z))^{2/p}}$ has almost everywhere boundary limits.  Hence so does  $\displaystyle{f(z)}$.  Additionally, $\displaystyle{\norm{f}_{\hardy^p, \mu} = \norm{f}_{p, \mu}}$. Furthermore, $\displaystyle{L^p(\T, \mu)}$  is a Banach space and $\displaystyle{\hardy^p(\T, \mu)}$ is a closed subspace of  $\displaystyle{L^p(\T, \mu)}$.

In particular, again when $p = 2$, we obtain the weighted Szeg\"o projection
\begin{displaymath}
\mathbf{S_{\mu}}: L^2(\T, \mu) \longrightarrow \hardy^2(\T, \mu).
\end{displaymath}
Following the similar theory, we note that $\mathbf{S_{\mu}}$ is an integral operator 
\begin{displaymath}
\mathbf{S_{\mu}}f(z) = \int_{\T} S_{\mu}(z, w) f(w) \mu(w) d\theta.
\end{displaymath}
If ~~$\{f_n(z)\}_{n = 0}^{\infty}$ is an orthonormal basis for $\hardy^2(\T, \mu)$ then 
\begin{displaymath}
S_{\mu}(z,w)=\sum_{n=0}^{\infty}f^{}_n(z)\overline{f^{}_n(w)}.
\end{displaymath}

We are interested in the action of $\Smu$ on $L^p(\mathbb{T},\mu)$. By definition, $\Smu$ is a bounded operator from $L^2(\mathbb{T},\mu)$ to $L^2(\mathbb{T},\mu)$. The problem we investigate is the boundedness of $\Smu$ from $L^p(\mathbb{T},\mu)$ to $L^p(\mathbb{T},\mu)$ for other values of $p\in(1,\infty)$. Note that for any given weight $\mu$ as above, we can associate an interval $I_{\mu}\subset(1,\infty)$ such that $\Smu$ is bounded from $L^p(\mathbb{T},\mu)$ to $L^p(\mathbb{T},\mu)$ if and only if $p\in I_{\mu}$. By definition, $2\in I_{\mu}$ and by duality and interpolation, $I_{\mu}$ is a conjugate symmetric interval around 2. Namely, if some $p_0>2$ is in $I_{\mu}$, so is $q_0$ where $\frac{1}{q_0}+\frac{1}{p_0}=1$. 

In the classical setting, i.e., $\mu\equiv 1$, the Szeg\"o projection operator is bounded from $L^p(\mathbb{T})$ to $L^p(\mathbb{T})$ for any $1<p<\infty$, see \cite[page 257]{ZhuBook}.

The purpose of this note is to construct weights $\mu$ on $\mathbb{T}$ for which the corresponding interval $I_{\mu}$ can be any open interval larger than $\{2 \}$ but smaller than $(1, \infty)$.

\begin{theorem}\label{S_0bounded}
For any given $p_0>2$, there exists a weight $\mu$ on $\mathbb{T}$ such that $I_{\mu}=(q_0,p_0)$ where $\frac{1}{q_0}+\frac{1}{p_0}=1$, i.e., the weighted Szeg\"o projection $\Smu$ is bounded on $L^p(\mathbb{T},\mu)$ if and only if $q_0<p<p_0$.
\end{theorem}

We prove this theorem similar to the proof of the analogous statement for weighted Bergman projections in \cite{Zeytuncu13} with modifications from Bergman kernels to Szeg\"o kernels. The main ingredient is the theory of $A_p$ weights on $\mathbb{T}$.

When the weighted Szeg\"o projection $\mathbf{S}_{\mu}$ is bounded on $\displaystyle{}$ $\displaystyle{L^p(\T, \mu)}$ for some $p$, 
one can identify the dual space of the weighted Hardy space $\displaystyle{\hardy^p\left( \T, \mu \right)}$.  However, when $\displaystyle{\mathbf{S}_{\mu}}$ fails to be bounded, a different approach is needed to identify the dual spaces.  In the third section, we address this issue and describe the dual spaces of weighted Hardy spaces.   

The following notations are used in the rest of the note. We denote by $\displaystyle{f(z) \simeq g(z)}$ when  $\displaystyle{c \cdot g(z) \leq f(z) \leq C \cdot g(z)}$ for some positive constants $c$ and $C$ which are independent of $z$.  Similarly we denote by  $\displaystyle{f(z) \lesssim g(z)}$ when  $\displaystyle{f(z) \leq C\cdot g(z)}$ for some positive constant $C$. We use $d\theta$ for the Lebegue measure on the unit circle $\mathbb{T}$. When we integrate functions (that are also defined on the unit disc) on $\mathbb{T}$, instead of writing $e^{i\theta}$, we keep $z$ and $w$ as the variables.

\section{Proof of Theorem \ref{S_0bounded}}

\subsection{Relation between weighted kernels} The particular choice of $\mu_{}(z)$ indicates the following relation between the weighted Szeg\"o kernels ${S_{\mu}}(z,w)$ and the ordinary Szeg\"o kernel ${S}(z,w)$.

\begin{proposition}\label{WeightedNonweighted}
For $\mu(z)=|g(z)|^2$ as above, the following relation holds
\begin{equation} \label{3.5}
S\left(z, w \right) = g_{}(z)  S_{\mu}\left( z, w \right) \overline{g_{}(w)}.
\end{equation} 
\end{proposition}

\begin{proof}
Let $\displaystyle{\{e_n(z)\}_{n=0}^{\infty}}$ be an orthonormal basis for $\hardy^2\left( \mathbb{T}\right)$. Since $g_{}(z)$ does not vanish inside $\mathbb{D}$, each $\frac{e_n(z)}{g_{}(z)}$ is a holomorphic function on $\disc$ and is in $\hardy^2(\mathbb{T},\abs{g_{}}^2)$ by construction.  Following the orthonormal properties of $\displaystyle{e_n(z)}$'s we have
\begin{displaymath}
\left<\frac{e_n(z)}{g(z)_{}},\frac{e_m(z)}{g(z)_{}}\right>_{\mu}=\left<e_n(z),e_m(z)\right> = \delta_{n,m},
\end{displaymath}
where $\displaystyle{}$ $\displaystyle{\delta_{n, m}}$ is the Kronecker delta.

Also for any $f$ in $\hardy^2(\mathbb{T},\abs{g_{}}^2)$, $(f\cdot g)_{}$ is in $\hardy^2(\mathbb{T})$ and hence can be written as a linear combination of the $e_n(z)$'s.  Consequently so can $f$ using $\frac{e_n(z)}{g(z)}$'s.  Hence,$\displaystyle{\left\{{e_n}(z)/{g(z)_{}}\right\}}_{n=0}^{\infty}$ is an orthonormal basis for $\hardy^2(\mathbb{T},\abs{g_{}}^2)$.

Therefore, using the basis representation of the Szeg\"o kernels we obtain
\begin{align*}
S(z,w)&=\sum_{n=0}^{\infty}e_n(z)\overline{e_n(w)}=g_{}(z)\left(\sum_{n=0}^{\infty}\frac{e_n(z)}{g_{}(z)}\overline{\frac{e_n(w)}{g_{}(w)}}\right)\overline{g_{}(w)}\\
&=g_{}(z)S_{\mu}(z,w)\overline{g_{}(w)}.
\end{align*}
\end{proof}

\subsection{$A_p$ weights on $\T$}

For $p \in (1, \infty)$, a weight $\mu$ on $\T$ is said to be in $A_p(\T)$ if

\begin{displaymath}
\underset{I \subset \T}{\underset{I}{\text{sup}}} \left(\frac{1}{\abs{I}}\int_{I} \mu(\theta) d\theta\right) \left( \frac{1}{\abs{I}}  \int_{I} \mu(\theta)^{\frac{-1}{p - 1}} d\theta\right)^{p-1} < \infty,
\end{displaymath}
where $I$ denotes intervals in $\T$.

These weights are used to characterize the $L^p$ regularity of the ordinary Szeg\"o projection on weighted spaces.  The following result appears in 
\cite{Garnett}  and is used in \cite[2.3]{LanzaniStein04} in connection with a conformal map based approach to the investigation of the unweighted Szeg\"o projection for a general domain.


\begin{theorem}\label{Ap}
The ordinary Szeg\"o projection $\mathbf{S}$ is bounded from $\displaystyle{L^p(\T, \mu)}$ to $\displaystyle{L^p(\T, \mu)}$ if and only if $\mu \in A_p(\T)$.
\end{theorem}

\begin{proof}

This result is an immediate consequence of the fact that the Szeg\"o kernel of the unit disc agrees with the Cauchy kernel (see \cite{KerzmanStein}) together with the classical weighted theory for the latter, see also \cite{Garnett}.

\end{proof}

The following theorem follows from Proposition \ref{3.5} and Theorem \ref{Ap}.

\begin{proposition}\label{Main1}
For $1 < p < \infty$ and $\mu(z)=|g(z)|^2$ as above the following are equivalent.
\begin{enumerate}
\item $\mathbf{S_{\mu}}$ is bounded from $\displaystyle{L^p(\T, \abs{g_{}}^2)}$ to    $\displaystyle{}$     $\displaystyle{L^p(\T, \abs{g_{}}^2)}$.
\item $\mathbf{S}$ is bounded from $\displaystyle{L^p(\T, \abs{g_{}}^{2 - p})}$ to    $\displaystyle{}$     $\displaystyle{L^p(\T, \abs{g_{}}^{2 - p})}$.
\item $\abs{g_{}}^{2 - p} \in A_p(\mathbb{T})$. 
\end{enumerate}
\end{proposition}

\begin{proof}
Theorem \ref{Ap} gives the equivalence of (2) and (3).  We show the equivalence of (1) and (2).  Using the relation between the kernels from the previous proposition, we obtain the following relation between the corresponding operators:
\begin{displaymath}
g_{}(z) \left({\bf{S}}_{\mu} f\right)(z) = \left({\bf S}(f\cdot g_{})\right)(z) \qquad \text{for} ~~ f \in L^2(\T, \abs{g_{}}^2).
\end{displaymath}

Indeed, suppose (2) is true.  Then
\begin{eqnarray*}
\norm{{\bf S}_{\mu} f}_{p, \abs{g_{}}^2}^p &=& \int_{\mathbb{T}} \abs{({\bf S}_{\mu} f)(w)}^p \abs{g_{}(w)}^2d\theta \\
&=& \int_{\mathbb{T}}\abs{({\bf S}(f\cdot g_{}))(w)}^p \abs{g_{}(w)}^{2-p}d\theta = \norm{{\bf S}(f\cdot g_{})}^p_{p, \abs{g_{}}^{2 - p}}\\
&\lesssim& \norm{f\cdot g}^p_{p, \abs{g_{}}^{2 - p}} = \norm{f}^p_{p, \abs{g_{}}^2}
\end{eqnarray*}
which proves (1).

Now when (1) is true,
\begin{eqnarray*}
\norm{{\bf S}_{} f}_{p, \abs{g_{}}^{2 - p}}^p &=& \int_{\mathbb{T}} \abs{({\bf S}_{} f)(w)}^p \abs{g_{}(w)}^{2 - p} d\theta\\
&=& \int_{\mathbb{T}}\abs{({\bf S}_{\mu}(f/g_{}))(w)}^p \abs{g_{}(w)}^2d\theta = \norm{{\bf S}_{\mu}(f/g_{})}^p_{p, \abs{g_{}}^2}\\
&\lesssim& \norm{f/g_{}}^p_{p, \abs{g_{}}^2} = \norm{f}^p_{p, \abs{g_{}}^{2 - p}}
\end{eqnarray*}
and hence (2) is true.  

\end{proof}

We can now present a family of weights that behave as claimed in Theorem \ref{S_0bounded}.

\begin{theorem}\label{Main2}
For $\alpha \geq 0 $, let
$g_{\alpha}(z) = (z - 1)^{\alpha}$ and $\mu_{\alpha}(z) =\abs{g_{\alpha}(z)}^2$. Then the weighted Szeg\"o projection operator $\mathbf{S}_{\mu_{\alpha}}$ is bounded on $L^p(\mathbb{T},\mu_{\alpha})$ if and only if $p \in \left( \frac{2\alpha + 1}{\alpha + 1}, \frac{2\alpha + 1}{\alpha} \right)$.
\end{theorem}

\begin{remark}
Theorem \ref{Main2} is a quantitative version of Theorem \ref{S_0bounded} and therefore we  obtain a proof of Theorem \ref{S_0bounded} when we prove Theorem \ref{Main2}.
\end{remark}

\begin{remark}
Note that as $\alpha\to 0^+$ the interval $\left( \frac{2\alpha + 1}{\alpha + 1}, \frac{2\alpha + 1}{\alpha} \right)$ approaches to $(1,\infty)$ and as $\alpha\to\infty$ the interval $\left( \frac{2\alpha + 1}{\alpha + 1}, \frac{2\alpha + 1}{\alpha} \right)$ approaches to $\{2\}$. Hence, any conjugate symmetric interval around $2$ can be achieved as the boundedness range of a weighted Szeg\"o projection.
\end{remark}

\begin{proof}
First note that on intervals $I$ with  $\theta = 0 \notin I$, the weight $\abs{g_{\alpha}(z)}^{2 - p}= \abs{z - 1}^{\alpha(2 - p)} \simeq C$.  Therefore, both integrals in  the $A_p(\T)$ condition are finite and hence so is the supremum over all such intervals when $p$ is in the given range.  On intervals that contain $z = 0$ we have the following. 

\subsubsection*{Step 1}
We show that for the weights $\omega(z) = \abs{g_{\alpha}(z)}^{2 - p} = \abs{z - 1}^{\alpha(2 - p)}$, the second integral in the $A_p(\T)$ condition diverges for arcs $I = (-\epsilon, \epsilon)$ if and only if  $p$ is outside the given region.  

For intervals $I = (-\epsilon, \epsilon)$ with small  $\epsilon$ and $p \leq  \frac{2\alpha + 1}{\alpha + 1}$, 
\begin{eqnarray*}
 \int_I \omega(z)^{\frac{1}{1 - p}} d\theta &=& \int_{-\epsilon}^{\epsilon} \abs{e^{i\theta} - 1}^{\frac{\alpha (2 - p)}{(1 - p)}} d\theta \\
 &=& \int_{-\epsilon}^{\epsilon} \left( \sqrt{2(1 - \cos(\theta))} \right)^{\frac{\alpha (2 - p)}{(1 - p)}} d\theta \\
 &\simeq&   \int_{-\epsilon}^{\epsilon}  \left( \theta \right)^{\frac{\alpha (2 - p)}{(1 - p)}} d\theta ~=~ \infty \\
\end{eqnarray*}
because \qquad $\displaystyle{\frac{\alpha(2 - p)}{(1 - p)} \leq -1}$.  
Hence $\omega \not \in A_p(\T)$ for such $p$.

Also  when $p \geq \frac{2\alpha + 1}{\alpha}$, 
\begin{eqnarray*}
 \int_I \omega(z) d\theta 
 &\simeq&   \int_{-\epsilon}^{\epsilon}  \left( \theta \right)^{\alpha (2 - p)} d\theta ~=~ \infty\\
\end{eqnarray*}
because $\displaystyle{\alpha(2 - p) \leq -1}$.
Hence $\omega \not \in A_p(\T)$ for  $p \geq \frac{2\alpha + 1}{\alpha}$ either.

The same calculations show convergence of all integrals for  $p$ in the desired range.

\subsubsection*{Step 2}

We show that for $p\in (\frac{2\alpha + 1}{\alpha + 1},\frac{2\alpha + 1}{\alpha})$ and any (general) interval $I = (\theta_0 - R, \theta_0 + R)$ with $\theta_0 \neq 0$ the integrals in the $A_p$ condition is finite.  We consider two cases.

\subsubsection*{Case 1}

$\displaystyle{I \cap \text{Arc}(0, 2R) = \emptyset}$.

On such intervals $I$, 
$\displaystyle{3R < \theta_0}$ ~and so~ $\displaystyle{2\theta_0/3  \leq \theta_0 - R \leq\theta \leq \theta_0 + R \leq 4\theta_0/3}$ ~giving ~ $\theta \simeq \theta_0$.  So, $\displaystyle{\omega = \abs{z - 1}^{\alpha(2 - p)} \simeq \theta_{0}^{\alpha(2 - p)}}$.

Therefore,

\begin{eqnarray*}
\frac{1}{\abs{I}} \int_I \omega(z) d\theta 
 &\lesssim&  \frac{1}{2R} \int_I   \theta_0^{\alpha(2 - p)} ~d\theta ~=~  \theta_0^{\alpha(2 - p)}. \\
\end{eqnarray*}
and 
\begin{eqnarray*}
\left( \frac{1}{\abs{I}} \int_I \left( \omega(z) \right)^{\frac{1}{1 - p}} d\theta \right)^{p - 1} 
 &\lesssim& \left(  \frac{1}{2R} \int_I \theta_0^{\frac{\alpha(2 - p)}{1 - p}} ~d\theta \right)^{p - 1}\\
 &=& \theta_0^{-\alpha(2 - p)}
\end{eqnarray*}

Hence the supremum over all such intervals is finite.

\subsubsection*{Case 2}

$\displaystyle{I \cap \text{Arc}(0, 2R) \neq \emptyset}$.

In this case,  since $\displaystyle{I \subset \text{Arc}(0, 4R)}$ and $\displaystyle{\alpha(2 - p) + 1 > 0}$ ~when~ $\displaystyle{\frac{2\alpha + 1}{\alpha} > p}$ we have,
 \begin{eqnarray*}
\frac{1}{\abs{I}} \int_I \omega(z) d\theta 
 &\simeq&  \frac{1}{8R}~ 2 \int_{0}^{4R}   \theta^{\alpha(2 - p)} ~d\theta ~=~  \frac{1}{4R}~ \frac{\theta^{\alpha(2 - p) + 1}}{[\alpha(2 - p) + 1]}\Big|_{0}^{4R}  = \frac{4R^{\alpha(2 - p)}}{\alpha(2 - p) + 1}.\\
\end{eqnarray*}

Also since   $\displaystyle{\frac{\alpha(2 - p)}{1 - p} + 1 > 0}$ ~~ when ~~ $\displaystyle{\frac{2\alpha + 1}{\alpha + 1} < p}$,
 \begin{eqnarray*}
\left( \frac{1}{\abs{I}} \int_I \omega(z)^{\frac{1}{1 - p}} d\theta \right)^{p - 1} 
 &\simeq& \left( \frac{1}{8R}~ 2 \int_{0}^{4R}   \theta^{\frac{\alpha(2 - p)}{1 - p}} ~d\theta \right)^{p - 1} ~=~  \left( \frac{1}{4R}~ \frac{\theta^{\frac{\alpha(2 - p)}{1 - p} + 1} }{\left[\frac{\alpha(2 - p)}{1 - p} + 1\right]} \right)^{p - 1}    \\
 &\simeq& \left(  \frac{R^{\frac{\alpha(2 - p)}{1 - p} } }{\left[\frac{\alpha(2 - p)}{1 - p} + 1\right]} \right)^{p - 1}    
 = \frac{2R^{- \alpha(2 - p)}}{\left[\frac{\alpha(2 - p)}{1 - p} + 1\right]^{p - 1}}.\\
\end{eqnarray*}

Therefore, the supremum over all intervals of the type in case two are also finite and  $\omega=|g|^{2 - p}\in A^p(\mathbb{T})$ if and only if $p\in (\frac{2\alpha + 1}{\alpha + 1},\frac{2\alpha + 1}{\alpha})$.

\end{proof}

\begin{remark}
The analog of Theorem \ref{S_0bounded} for domains in $\mathbb{C}^n$ ($n\geq 2$) is an open problem. See \cite{BekolleSzego} for a partial result. Also see \cite{LanzaniStein13} for the regularity on strongly pseudoconvex domains.
\end{remark}


\section{Duality}

In this section, we investigate the duals of Hardy spaces corresponding to weights from the previous section.
For $\alpha \geq 0$ and $\displaystyle{\mu_{\alpha}(z) = \abs{z - 1}^{2\alpha}}$, a consequence of Theorem \ref{Main2} is the following.

\begin{theorem}\label{FixedWeightDualHardy}
Let $\alpha \geq 0$ and $\displaystyle{\mu_{\alpha}(z) = \abs{z - 1}^{2\alpha}}$.  For any $\displaystyle{p \in \left( \frac{2\alpha + 1}{\alpha + 1}, \frac{2\alpha + 1}{\alpha}  \right)}$, the dual space of the weighted Hardy space $\displaystyle{}$ $\displaystyle{\hardy^p\left( \T, \abs{z - 1}^{2\alpha}   \right)}$ can be identified by $\displaystyle{\hardy^q\left( \T, \abs{z - 1}^{2\alpha}   \right)}$ where $\displaystyle{\frac{1}{p} + \frac{1}{q} = 1}$ and under the pairing 
\begin{displaymath}
\left< f, h  \right> = \int_{\T} f(z) \overline{h(z)} \abs{z - 1}^{2\alpha} d\theta.
\end{displaymath}
\end{theorem}

\begin{proof}
This is a standard argument; however, we present a proof here for completeness.  
For a given function $\displaystyle{}$ $\displaystyle{h \in \hardy^q\left( \T, \abs{z - 1}^{2\alpha}  \right)}$, we define a linear functional on $\displaystyle{\hardy^p\left( \T, \abs{z - 1}^{2\alpha}  \right)}$ by 
\begin{displaymath}
{\bf{G}}(f) = \int_{\T} f(z) \overline{h(z)} \abs{z - 1}^{2\alpha} d\theta. 
\end{displaymath}
It is clear that by the H\"older's inequality, $\G$ is a bounded functional with the operator norm less than $\displaystyle{\norm{h}_{\hardy^q\left( \T, \abs{z - 1}^{2\alpha} \right)}}$.

For the opposite direction, let $\G$ be a bounded linear functional on $\displaystyle{\hardy^p\left( \T, \abs{z - 1}^{2\alpha} \right)}$.  By the Hahn-Banach theorem, $\G$ extends to a bounded linear functional on $\displaystyle{L^p\left( \T, \abs{z - 1}^{2\alpha} \right)}$. 
Now using the duality of $L^p$ spaces, we find a function $\displaystyle{h \in L^q\left( \T, \abs{z - 1}^{2\alpha} \right)}$ such that
\begin{displaymath}
\G(f) = \int_{\T} f(z) \overline{h(z)} \abs{z - 1}^{2\alpha} dz \qquad \text{for} \qquad f \in L^p\left( \T, \abs{z - 1}^{2\alpha} \right).
\end{displaymath}
When we restrict $\G$ onto $\displaystyle{L^p\left( \T, \abs{z - 1}^{2\alpha}  \right) \cap \hardy^2\left( \T, \abs{z - 1}^{2\alpha}  \right)}$
and use self-adjointness of $\displaystyle{{\bf{S}}_{\mu_{\alpha}}}$ we get the following. 
\begin{eqnarray*}
\G(f) &=& \int_{\T} f(z)~\overline{h(z)} \abs{z - 1}^{2\alpha} d\theta\\
&=& \int_{\T} ({\bf{S}}_{\mu_{\alpha}}f)(z) ~\overline{h(z)} \abs{z - 1}^{2\alpha}d\theta\\
&=& \int_{\T} f(z) ~\overline{({\bf{S}}_{\mu_{\alpha}}h)(z)} \abs{z - 1}^{2\alpha} d\theta
\end{eqnarray*}
 for $\displaystyle{f \in L^p\left( \T, \abs{z - 1}^{2\alpha}  \right) \cap \hardy^2\left( \T, \abs{z - 1}^{2\alpha}  \right)}$.
 
Since the intersection of these two spaces is dense in  $\displaystyle{}$   $\displaystyle{\hardy^p\left( \T, \abs{z - 1}^{2\alpha} \right)}$, we note that $\G$ is represented by the function  $\displaystyle{({\bf{S}}_{\mu_{\alpha}}h)(z)}$ and  $\displaystyle{{\bf{S}}_{\mu_{\alpha}}h \in \hardy^q\left( \T, \abs{z - 1}^{2\alpha} \right) }$ by Theorem \ref{Main2}.
\end{proof}

A natural question arises after this statement.  How can we identify the dual space of the weighted Hardy space, $\displaystyle{\hardy^p\left( \T, \abs{z - 1}^{2\alpha} \right)}$, when $\displaystyle{p \notin \left( \frac{2\alpha + 1}{\alpha + 1} , \frac{2\alpha + 1}{\alpha} \right)}$? 
The answer to this question follows from the following result on the boundedness of the weighted Szeg\"o projection, $\displaystyle{{\bf{S}}_{\mu_{\alpha}}}$. 
Similar results for weighted Bergman projections have been presented recently in \cite{ArroussiPau13} and \cite{ConstantinPelaez13}.

\begin{proposition}\label{ChangeScaleAndMeasure}
Let $\displaystyle{}$ $\displaystyle{\alpha \geq 0}$ and $\displaystyle{\mu_{\alpha} = \abs{z - 1}^{2\alpha}}$.  For any $\displaystyle{1 < p < \infty}$, the weighted Szeg\"o  projection $\displaystyle{{\bf{S}}_{\mu_{\alpha}}}$ is bounded on $\displaystyle{L^p\left( \T, \abs{z - 1}^{\alpha p} \right)}$.    
\end{proposition}

\begin{remark}
Note that as $p$ varies, changes occur not only the integrability scale but also in the measure.
\end{remark}

\begin{proof}
The proof follows from the relation between the kernels in Proposition \ref{WeightedNonweighted} and the fact that the unweighted Szeg\"o projection $\textbf{S}$ is bounded on $\displaystyle{L^p(\T)}$ for $1<p<\infty$.

Let us take $\displaystyle{f(z) \in L^p\left( \T, \abs{z - 1}^{\alpha p}  \right)}$ and set $$\widetilde{f}(z)=f(z)  \frac{\abs{z - 1}^{2\alpha}}{(z-1)^{\alpha}},$$ then we have $\widetilde{f} \in L^p(\T)$. Using this notation, we notice 
\begin{eqnarray*}
{\bf{S}}_{\mu_{\alpha}}f (z)&=& \int_{\T} S_{\mu_{\alpha}}(z, w) f(w)\abs{w - 1}^{2\alpha}~ d\theta \\
&=& \frac{(z - 1)^{\alpha}}{(z - 1)^{\alpha}} ~\int_{\T} S_{\mu_{\alpha}}(z, w)  (w-1)^{\alpha} f(w)\frac{\abs{w - 1}^{2\alpha}}{(w-1)^{\alpha}}  ~ d\theta\\
&=& \frac{1}{(z - 1)^{\alpha}} ~\int_{\T} S(z, w) \widetilde{f}(w) ~ d\theta\\
&=& \frac{1}{(z - 1)^{\alpha}} \textbf{S}\left(\widetilde{f}(w) \right)(z),
\end{eqnarray*}
where we invoke Proposition  \ref{WeightedNonweighted} when we pass from the second to the third line. Next by using the fact that the unweighted Szeg\"o projection operator $\textbf{S}$ is bounded on $\displaystyle{L^p(\T)}$, we obtain the following.
\begin{eqnarray*}
\norm{{\bf{S}}_{\mu_{\alpha}} f}_{L^p\left( \T, \abs{z - 1}^{\alpha p}  \right)} &=& \int_{\T} \abs{z - 1}^{\alpha p} \frac{1}{\abs{z - 1}^{\alpha p}} ~~ \left\vert 
\textbf{S}\left(\widetilde{f}(w) \right)(z)
 \right\vert^p d\theta\\
&=& \norm{
\textbf{S}\left(\widetilde{f}(w) \right)
}^p_{L^p(\T)}\lesssim\norm{
\widetilde{f}(w)
}^p_{L^p(\T)}\\ &=& \norm{ f(w) \frac{\abs{w - 1}^{2\alpha}}{(w-1)^{\alpha}}}^p_{L^p(\T)}\\
&=& \norm{f}^p_{L^p\left( \T, \abs{z - 1}^{\alpha p}  \right)}.
\end{eqnarray*}
This finishes the proof of the proposition.
\end{proof}

Now we can answer the duality question by using Proposition \ref{ChangeScaleAndMeasure}.  Following the same argument in the proof of Theorem \ref{FixedWeightDualHardy} we obtain the following statement.

\begin{theorem}\label{Main3}
Let $\displaystyle{}$  $\displaystyle{\alpha \geq 0}$ and $\displaystyle{\mu_{\alpha} = \abs{z - 1}^{2\alpha}}$.  Then for any $\displaystyle{p \in (1, \infty)}$, the dual space of the weighted Hardy space $\displaystyle{\hardy^p\left( \T, \abs{z - 1}^{\alpha p}  \right)}$ can be identified by $\displaystyle{\hardy^q\left( \T, \abs{z - 1}^{\alpha q} \right)}$ where $\displaystyle{\frac{1}{p} + \frac{1}{q} = 1}$, under the pairing 
\begin{displaymath}
\left<  f, h \right> = \int_{\T} f(z) \overline{h(z)} \abs{z - 1}^{2\alpha} ~ d\theta.
\end{displaymath}
\end{theorem}

At first, the two duality results in Theorem \ref{FixedWeightDualHardy} and Theorem \ref{Main3} may seem confusing for 
$\displaystyle{p \in \left(\frac{2\alpha + 1}{\alpha + 1}, \frac{2\alpha + 1}{\alpha}   \right)}$.  However, the main point is to note the difference in the exponents of the weights and the way the pairing is defined.  We illustrate these two results in the following example.

\begin{example}
Let us take $\alpha = 1/2$.  Then $\displaystyle{{\bf{S}}_{\abs{z - 1}}}$ is bounded on $\displaystyle{L^p\left( \T, \abs{z - 1}  \right)}$ for $\displaystyle{p \in (4/3, 4)}$.  In particular, for any $\displaystyle{p \in (4/3, 4)}$, the dual space of $\displaystyle{\hardy^p\left( \T, \abs{z - 1}  \right)}$ can be identified by $\displaystyle{\hardy^q\left( \T, \abs{z - 1}  \right)}$ where $\displaystyle{\frac{1}{p} + \frac{1}{q} = 1}$, under the pairing
\begin{displaymath}
\left< f, h \right>_{\abs{z - 1}} = \int_{\T} f(z)~ \overline{h(z)}~ \abs{z - 1} ~ d\theta.
\end{displaymath}

On the other hand, using the second duality result for any $p > 1$, the dual space of $\displaystyle{\hardy^p\left( \T, \abs{z - 1}  \right)}$ can be identified by $\displaystyle{\hardy^q\left( \T, \abs{z - 1}^{q/p} \right)}$ when $\displaystyle{\frac{1}{p} + \frac{1}{q} = 1}$, under the pairing 
\begin{displaymath}
\left< f, h \right>_{\abs{z - 1}^{2/p}} = \int_{\T} f(z)~ \overline{h(z)}~ \abs{z - 1}^{2/p} ~ d\theta.
\end{displaymath}
\end{example}

\section{Acknowledgment} 

We thank the anonymous referee for constructive comments on the proofs of Proposition 3 and Theorem 5 and also for the useful editorial remarks on the exposition of the article. This study started when the first author visited the second author at Texas A\&M University, we thank the Department of Mathematics for the support.


\end{document}